\documentclass[12pt,reqno]{amsart}

\usepackage{amsmath,amsthm,amssymb,comment,fullpage}

\usepackage{times}
\usepackage{latexsym}
\usepackage[T1]{fontenc}

\usepackage{graphicx}
\usepackage{epstopdf}

\usepackage{mathrsfs}

\usepackage{latexsym}

\usepackage{epsfig}

\usepackage{amsmath,amsfonts,amsthm,amssymb,amscd}

\input amssym.def

\input amssym.tex

\usepackage{color}

\usepackage{hyperref}

\usepackage{url}

\usepackage{breakurl}

\newcommand{\bburl}[1]{\textcolor{blue}{\url{#1}}}

\numberwithin{equation}{section}

\newtheorem{thm}{Theorem}[section]

\newtheorem{cor}[thm]{Corollary}

\theoremstyle{plain}

\newtheorem{definition}[thm]{Definition}

\newtheorem{example}[thm]{Example}

\newtheorem{lemma}[thm]{Lemma}

\newtheorem*{theorem*}{Theorem}

\newtheorem{case}[thm]{Case}

\newtheorem{remark}[thm]{Remark}

\newcommand\N{\mathbb{N}}
\newcommand\Z{\mathbb{Z}}
\newcommand\R{\mathbb{R}}

\begin{document}
\title[Digit sequences of rational billiards on tables which tile $\R^2$.] 
{Digit sequences of rational billiards on tables which tile $\R^2$}
\author{Corey Manack and Marko Savic}
\begin{abstract}
We classify the periodic digit strings which arise from periodic billiard orbits on the four convex $n$-gons $\Delta$ which tile $\R^2$ under reflection, answering problem a posed in \cite{BU08}. $\Delta$ is either an equilateral triangle, a $45-45-90$ triangle, a $30-60-90$ triangle, or a rectangle.
\end{abstract}
\maketitle
\section{Overview}\label{S:overview}
\subsection{Introduction}
Our introduction follows \cite{BU08}. The trajectory $t$ traced out by a billiard ball on a frictionless billiards table is completely determined by its initial position $p_0$ and direction $v_0$. When the ball strikes a rail it bounces in such a way that its angle of incidence equals its angle of reflection. We seek to understand the trace of a billiards trajectory when the billiard table $\Delta$ is a convex $n$-gon $\Delta$ which tiles $\R^2$ under reflection. These are: an equilateral triangle, a $45-45-90$ triangle, a square, and a $30-60-90$ triangle with reflection groups the affine Weyl groups $A_2,B_2,D_2, G_2$ (cf. e.g. \cite{HW88}) respectively.   
These shapes appear as a category of solutions to stamp folding problems \cite{KU11}. If we arbitrarily label the edges of $\Delta$ as $0,1,\ldots,n-1$ then from any trajectory $t$ one may record the sequence of digits $\mod\ n$ from successive bounces of a billiards trajectory. We call any string generated by a billiard trajectory a \emph{billiard sequence} (or billiard string). A trajectory is \emph{periodic} if it returns to its initial position and direction in finite time; otherwise it's called an \emph{aperiodic} trajectory. Obviously periodic trajectories yield periodic sequences. Baxter and Umble \cite{BU08} posed the following question:
\begin{equation}
\label{equ:researchq}
\text{Which digit sequences arise from periodic billiard trajectories on $\Delta$?\footnote{the original question in \cite{BU08} was stated when $\Delta$ is an equilateral triangle.}} 
\end{equation}
See \cite{Bo15} for an overview when $\Delta$ is an equilateral triangle. We provide a complete solution to \ref{equ:researchq} for the four tables listed above. The paper is organized as follows: preliminaries \S\ref{sec:prelims} introduces notation and outlines the strategy for the classification which is carried out in \S\ref{sec:prelims}. The main result (Theorem \ref{thm:Umbledone}) is an reduction of the classification to partitions of the multiples of $q$ mod $n$. Theorem \ref{thm:strED} provides a method of construction for billiard strings in terms of the Euclidean division algoritm. One should think of the Theorems  \ref{thm:strED} \ref{thm:Umbledone} as the an understanding of billiard sequences in terms of the initial direction $v_0$. Section \S\ref{sec:translation} provides an elegant relationship between all billiard sequences of fixed direction $v_0$ (allowing $p_0$ to vary).  We then make some general comments for aperiodic digit sequences in \S\ref{sec:aperiodic}. 
\section{Preliminaries}
\label{sec:prelims}
\noindent To classify billiard sequences employ the standard trick of unfolding in polygonal billiards; by this we mean:    
\begin{enumerate}
	\item $\Delta$ unfolds along its edges of reflection to create a tesselation $\Lambda\subset\R^2$, whose vertices $V(\Lambda)$ forms a rank two lattice in $\R^2$, and 
	\item A given billiard trajectory $t$ may be unfolded along its edges of reflection (in $\Delta$) at each points of contact with the edges of $\Delta$ to yield a ray $\tau$ through $\R^2$.
\end{enumerate} 
It is on this very special collection of tables where the billiards table $\Delta$ unfolds into a convex uniform lattice under reflection. Moreover, $\Lambda$ has a canonical choice for edge labels (once the rails of $\Delta$ ar labelled). Precisely, the edges are the interiors of line segments which connect a vertex in $\Lambda$ to any of its nearest neighbors. \footnote{$\mathcal{T}$ has the structure of a moduli space but this level of sophistication is not required to answer \ref{equ:researchq}.}. Each edge inherits the canonical label from $\Delta$ under reflection.  If we write $E(\Lambda)\subset \Lambda$ for the collection of edges then $\Lambda= V(\Lambda)\sqcup E(\Lambda)$. Conversely, $\tau\subset\Lambda$ may be folded along the edges of $\Lambda$ to produce a billiards orbit $t\in\Delta$. 
\begin{definition}
	Call $\tau$ an \emph{unfolding} of $t$ and $t$ the \emph{folding} of $\tau$.
\end{definition}
\begin{remark}
We reserve the terminology ``trajectory'' or ``orbit'' to mean either $t$ or $\tau$ and write $\tau_s$ for the digit string generated by the intersections of $\tau$ with the edges of $\Lambda$, $t_s$ for the digit string generated by bounces of $t$ with the edges of $\Delta$.  Clearly $\tau_s=t_s$ if $\tau$ is an unfolding of $t$. 
\end{remark}  
The next theorem coincides with our geometric understanding of a $\tau\in\mathcal{T}$ with $\tau_s$ periodic.
\begin{thm}
	A billiards trajectory $t$ is periodic if and only if there is an element $\lambda \in \Lambda$ for which $\tau+\lambda\subset \tau$ 
\end{thm}
Not every periodic $n$-nary sequence is a billiard sequence. We shall see that billiard sequences fall into a ``small'' subset of $\{0,1,\dots, n-1\}^{\infty}$. Denote by $\mathcal{T}$ the space of all billiard trajectories $t\in\Delta$, $\mathcal{T}_s$ the space of all billiard sequences.
Since we deal mostly with with unfoldings, we will abuse notation and write $\mathcal{T}$ for the space of all rays in $\R^2$
\begin{definition}
	A billiard trajectory $t\in\mathcal{T}$ \emph{singular} if $t$ is a folding of a ray $\tau\in\R^2$ which passes through a vertex of $\Lambda$.
\end{definition}
Singular trajectories have no well-defined billiard sequence but play a key role in the classification of translation classes as follows. For any vertex $\lambda\in\Lambda$, the continuous family of translated billiard sequences $\tau+s\lambda, s\in[0,1]$ remains invariant until $\tau$ passes through a vertex of $\Lambda$. Since the billiard strings between any pair of rays in $\tau+\Lambda$ differ by cyclic permutation and $\R^2/\Lambda$ is compact, the number of elements in $[\tau]$ is finite for any periodic $\tau$. Notice that for any nonsingular $\tau$, a small perturbation of $\tau$ leaves $\tau_s$ invariant. Hence, $\tau_s$ encodes the direction$v_0$ but not position $p_0$ of the orbit, prompting the next definition. 
\begin{definition}
	A \emph{translation class} $[\tau]\subset \mathcal{T}$ to be the set of all billiard strings that can be brought into $\tau$ by translation and $G$-equivalence.
\end{definition}
Notice that if $\tau_1,\tau_2\in [\tau]$ are two orbits in the same translation class then the corresponding foldings of $t_1,t_2$ are parallel at some interval during its orbit. The remainder of the paper will be spent classifying billiard strings in each of the four tables and then we shall organize them into translation classes.
\section{Classification of digit billiard strings for each of the four types}
\label{sec:calcs}
\subsection{Unfolding.}
In each of the four tables, $\Delta$ unfolds along its edges of reflection to create a tesselation $\Lambda\subset\R^2$ by reflected copies of $\Delta$. The symmetry group $G=S(\Lambda)$ of $\Lambda$ is an affine Weyl group. $G$ may be described by compositions a point group $K\subset G$ and translations $V(\Lambda)$, with $\R^2/S(\Lambda) = \Delta$. The $H$-action on $\Delta$ permutes the edge labels $0,1,\ldots, n-1$ which induces the obvious $H$-actions on $\Lambda$, $\mathcal{T}$ and $\mathcal{T}_s$ respectively. The vertices of $\Lambda$ are generated by two independent directions $v_1,v_2\in \R^2$. The following observation is key. There are two essential translation subgroups $H\unlhd H_u$ of $G$, one each for the labelled and unlabelled edges of $\Lambda$. Notice that $H$ has finite index in $H_u\cong V(\Lambda)$, and that $\R^2/H$ is a torus, tiled by $|H_u:H|$ labelled copies of $\Delta$ (see figures ,\ref{fig:MSA2},\ref{fig:MSB2},\ref{fig:MSD2} and \ref{fig:MSG2} for examples)). We adopt the convention that $\tau$ is a convex combination of the directions $v_1,v_2$:
\begin{equation}
\label{equ:initdir}
v_0=cv_1+(1-c)v_2, a\in [0,1]
\end{equation}
Then $\tau$ may be viewed a winding line $\tau/H$ on the torus $\R^2/H$. Therefore $t,\tau$ and $t_s=\tau_s$ are periodic if and only if $c$ is rational. Write $c= q/n$ where $q,n$ are coprime integers with $0<q\leq n/2$. Thus, our initial direction has the form
\begin{equation}
\label{equ:initdir2}
v_0=n^{-1}(qv_1+(n-q)v_2)
\end{equation} For convention, label $\Delta$ and select $p_0$ so that $\tau$ is a nonsingular ray starting on edge $0$ and each term in the sequence
\begin{equation}
\label{equ:stepseq}
B=p_0,p_0+v_0,p_0+2v_0,p_0+3v_0,\ldots
\end{equation}
generates a new digit of $t_s=\tau_s$. Assuming that no edge lies between consecutive terms of $B$, $t_s=\tau_s$ may be constructed from the sequence $B$ of points in \eqref{equ:stepseq}. As $v_1,v_2$ generate the vertices of $\Lambda$, a unit step in $v_1$ adds $j$ to the current digit (mod $n$) and a unit step in direction $v_2$ adds $i$ to the current digit (modulo the base). Therefore, the digits of $\tau_s$ are generated by the orbit of $kq/n \mod 1$, or $kq \mod n$, as $p_0$ may be taken to be negligably small $a>0$ in the $v_1$ coordinate. Thus the sequence \ref{equ:stepseq} generates the next digit by deciding whether the $k$th digit is $+i$ or $+j$ the $k-1$st as according to whether $a+kq \mod n < a+(k-1)q\mod n$. If the step from $p_0+(k-1)v_0$ to $p_0+kv_0$ exceeds the next highest multiple of $v_1$ then the $k$th digit is $+i$ the $k-1$st digit of $\tau_s$, otherwise it is $+l$.  Thus, we examine the forward orbit $a+kq \mod n, k\in\N$, then show how the four cases $A_2,B_2,D_2,G_2$ are simply a selection of the generation rule $i,j$ outlined above. 
\subsection{q-cycles}
\label{sec:qcycles}
For this section, let $n,q\in\N,\ n>q$ be fixed but arbitrary. Under the assumptions of \S\ref{sec:calcs} the terms of the forward orbit $a+kq \mod n, k\in\N$ determine the billiard sequence $\tau_s$. Obviously this sequence repeats after $n$ steps as to be expected by a periodic $\tau$. By long division, $n=bq+r$ for some $b>0$, $0\leq r<b$. 
\begin{definition}
For a remainder $0\leq a < q$ and integer $n$ there is a largest $k\in \N$ satisfying $a+kq<n$. Define the $q$-cycle (or simply cycle) of $a$ in $n$, $C_a$, to be the arithmetic progression $a,a+q,a+2q,\ldots,a+kq$. Call $a$ the \emph{minimal element} of $C_a$. Define the $q$-cycle length to be $|C_a|=k$ and if $|C_a|=k$ call $a+kq$ the maximal element of $C_a$. 
\end{definition}
It is evident from this definition that the set of distinct $q$ cycles of $n$, $\mathcal{C}_{q,n}$ (or $\mathcal{C}$ when $q,n$ are implied), form an ordered partition of $\Z/q\Z$. The minimal elements of $C_a$ appear consecutively within this list as an arithmetic progression of length $|C_a|$ that begins with minimal element $a$. Obviously, the first $q$-cycle is $C_0$ with minimal element $0$. 
Since the minimal element $a$ must satisfy $0\leq a<q$ there are $q$ $q$-cycles in total. The remainder $r$ determines cycle lengths: the cycles $C_0,\ldots,C_{r-1}$ have length $b+1$, whereas the cycles $C_r,\ldots ,C_{q-1}$ have length $b$. Notice that there are only two possible cycle length and the sum of the lengths over all $q$-cycles is $(b+1)r+b(q-r)=bq+r=n$. The next lemma expands on the structure of $q$ cycles. 
\begin{lemma}
\label{lem:minelet}
If $C_a$ is the $q$-cycle with maximal element $a+kq$, then $x:=a+(k+1)q-n$ is the minimal element of the cycle $C_x$. Specifically,
 \[
   x= \left\{
     \begin{array}{lr}
       a+q-r, &  1<a<r\\
       a-r, &  r\leq a < q
     \end{array}
   \right.
 \]
\end{lemma}
\begin{proof}
The proof is evident from routine calculations which are split into cases based on the length of $|C_a|$:
\begin{case}
$0\leq a < r$.
\end{case}
Since $|C_a|=b+1$, $a+bq$ is the $b+1$st (maximal) element of the $q$-cycle, whence $n\leq a+(b+1)q= a+bq+q$. Substituting $bq=n-r$ yields $n\leq a+(n-r)+q$, or $0\leq a+q-r$. From the assumption $a < r$ we get $a+q-r<q$. Thus, $a+(b+1)q-n=a+q-r$ is the minimal element of $C_{a+q-r}$.
\begin{case}
$r \leq a < q$. 
\end{case}
Since $|C_a|$ has length $b$, $a+(b-1)q$ is the maximal element of the $q$-cycle $n\leq a+bq$. As $bq=n-r$, $n\leq a+(n-r)$, so $0\leq a-r$. By assumption, $a<q$, so $a-r<q-r<q$. Thus $a+bq-n=a-r$ is the minimal element of $C_{a-r}$.
\end{proof}
From a $(q,n)$-cycle $\mathcal{C}_{n,q}$ one may produce a digit sequence by replacing the cycle $C_x\in \mathcal{C}_{n,q}$ with $|C_x|$ $+i$'s followed by one $+j$ as determined by the $q$-cycle $x,x+q,x+2q,\ldots x+(|C_x|-1)q$. We call this the ``$+i$ within, $+j$ between'' rule.  Starting with digit zero, the $q$-cycle rule shows us how to generate digits of an $m$-ary string.
\begin{thm}
\label{thm:Umbledone}
Let $d$ be a periodic string $d\in\{0,\ldots,m\}^{\infty}$. The following are equivalent
\begin{enumerate}
\item $d$ is a periodic nonsingular billiard sequence, i.e., the digits of $d=\tau_s$ for some $\tau\in \mathcal{T}$ 
\item The digits of $d$ are constructed from a $(q,n)$ string $\mathcal{C}$ for some $q,n$ coprime, according to the ``$+i$ within $+j$ between'' rule. 
\end{enumerate}
\end{thm}
In particular, the $(q,n)$ $\mathcal{C}$ provides the rule for constructing a billiard string $\tau_s$ in direction $v_0=n^{-1}(qv_1+(n-q)v_2)$ \ref{equ:initdir2}.     
\begin{proof}
$(1)\implies (2)$. Much of the argument collects previous observations. An $m$-ary Billiard sequence is constructed from $\tau = p_0+tv_0$ where
\[v_0=n^{-1}(qv_1+(n-q)v_2)\] where $q,n$ coprime as in \ref{equ:initdir2}. By \ref{equ:stepseq} the digits of $d$ are produced by the sequence of numerators in the $v_1$ coordinate of
\[p_0,p_0+v_0,p_0+2v_0,p_0+3v_0,\ldots\] mod $n$. These numerators precisely fall into the $(q,n)$-cycle structure $\mathcal{C}_{q,n}$.  The ``$+1$ within $-1$ between'' rule for $q$-cycles was defined as to make digit strings coincide.\\
$(1)\implies (2)$ Suppose that the $m$-ary digit string constructed from the $(q,n)$-cycle.  By Lemma \ref{lem:minelet}, $x+(|C_x|)q \mod n$ is the minimal element of the next cycle, from which we can inductively build the sequence of numerators $v_1$. Since $v_0$ is a convex combination of $v_1$ and $v_2$, $v_0$ as in \ref{equ:initdir2} are determined. As $V(\Lambda)$ is discrete, we may choose $p_0=av_1+(1-a)v_2$ with small enough $v_1$ coordinate so that
\begin{itemize} 
\item the correponding $q$-cycles starting with minimal elements $0$ or $a$ produce the same digit sequence in the ``$+i$ within $-j$ between'' rule, 
\item the ray $p_0+tv_0$ is nonsingular, and 
\item the sequence
\[p_0,p_0+v_0,p_0+2v_0,p_0+3v_0,\ldots\] produces the same digit string $d$ as the two $(q,n)$-cycles.
\end{itemize}
The theorem is proved.
\end{proof}
The essence of the reseach question \ref{equ:researchq} is method by which one can verify
whether a digit periodic string $d$ is a billiard sequence. By theorem \ref{thm:Umbledone}, this amounts to showing the digits fall into a ``$+i$ within $+j$ between'' construction from some $(q,n)$ cycle. We begin with a fairly brute force algorithm, then provide a more elegant means of detecting billard strings in section \S\ref{sec:strqcyc}. While $q,n$ are not known from the outset, we a priori write $n=bq+r$ and describe an algorithm to determine whether $d$ is a billiard string by producing $n,b,q,$ and $r$. We keep the assumption that $q<n/2$ so there are more +i's that +j's.  Recall that there are $q$ $q$-cycles and hence (at least) $q$-locations within digit string $d$ where a $+j$ increment occurs. So there must be $n-q$ locations within the first $n$ where a digit $+i$ occurs. As always we assume the first digit of $d$ is $0$. After $n$ digit changes The $n+1$st digit is $1(n-q)+(-1)q=n-2q=-2q \mod n$. If the $n+1$st digit is $0$ then the digit string repeats and $n$ emerges as the period. If $n-2q\neq 0 \mod n$ then the pattern of digit changes repeats but with a new leading digit. In this way $d$ has period $mn$ (if the string is $m$-ary) where a left shift by $n$ digits of $d$ affects a $\pm k$ action on each digitdigit, where $\pm k$ is the $n+1$st digit since the rule $+i,+j$ rule repeats. Take the first $n$ digits in either case. Within the $q$-cycle $C_a$ with minimal element $a<n$, subsequent digits are $+i$ the preceding digit $\mod n$ until $a+k'q>n$ then the subsequent digit is $+j$ the preceeding digit $\mod n$. The length $b+1$ therefore appends $b$ digits in $+i$ succession $\mod n$ followed by one digit $+j$, and $a+k'q \mod n$ is the minimal element of the next $q$-cycle by Lemma \ref{lem:minelet}.
To summarize the algorithm:
\begin{enumerate}
\item Find the period $p$ of $d$
\item If the string is $m$-ary and $m|p$ check whether the +i,+j pattern repeats p/m times. If so, consider the first $p/m$ digits and repeat step $2$ until both the $+i,+j$ rule and the digits themselves repeat.
\item Find the two cycle lengths $b,b+1$
\item Use long division to find $q,r$
\item Verify that the lengths of the $q$ cycles fall into the orbit of $-kr \mod q$. With exactly $r$ blocks of length $b+1$ and $q-r$ blocks of length $b$.
\item If yes to all the above, then $d$ is a given by a billiard orbit.  
\end{enumerate}
\subsection{Structure of translation classes}
\label{sec:strqcyc}
Let $n=bq+r, 0\leq r <q$ be as above. Theorem \ref{thm:Umbledone} shows that the digit sequences generated by rational nonsingular billiard orbits and $(q,n)$-cycles are equal up to suitable choice of $p_0$. However a billiard sequence determines only the direction and not the initial position. It is when a translated $\tau$ meets a vertex (i.e., bcomes singular) that we see a change in the digits sequence. So we seek to understand the family billiard sequences within a translation class $[\tau]\subset\mathcal{T}$. There are two main results in this section: 
\subsubsection{Recursive construction of billiard sequences in terms of the division algorithm.} Consider the sequence $C_{x_1}, C_{x_2},\ldots, C_{x_{q-1}}$ of $q$-cycles of $\mathcal{C}$ which partitions the orbit $0,q,2q,\ldots,(n-1)q$ mod $n$. Within this progression, the maximal element of $C_{x_i}$ is followed by the minimal element $C_{x_{i+1}}$. By Lemma \ref{lem:minelet} (with $a=0$) the $k$th minimal element is $cq-kr$ for some $c\in \N$. Modulo $q$, to the orbit of 
\[0,-r,-2r,\ldots -(q-1)r \mod q\]
records the minimal elements within $\mathcal C$ in order.
In other words, the $q$ cycles of fall into the sequence 
\[\mathcal{C}=(C_{0}, C_{-r},\ldots, C_{-(q-1)r}).\]
The lengths of each $q$ cycle is determined by the residue of $-kr$ mod $q$. This affords a recursive construction of $\mathcal{C}_{n,q}$ which we record as a theorem.
\begin{thm}
\label{thm:strED}
Write $n=bq+r$. The $(n,q)$-cycle $\mathcal{C}_{n,q}$ may be constusted from $\mathcal{C}_{q,-r}$ by replacing, for each element $y$ in the $-r$-cycle $C_{-kr}\in \mathcal{C}$ the $q$ -cycle $C_{y}$ of $n$. Therefore the digit strings built from $(n,q)$-cycles may be expressly constructed from $\mathcal{C}_{q,-r}$.
\end{thm}
In other words, $(n,q)$ cycles may be constructed according to the division algorithm
\section{Relation between elements in a translation class $[\tau]$} 
\label{sec:translation} 
The next lemma describes the lengths of $q$-cycles in terms of the consecutive minimal elements.  
\begin{lemma}
\label{lem:strlen}
Let $n=bq+r, 0\leq a\leq q-1$ as above. Then
 \[
   |C_{a+1}| = \left\{
     \begin{array}{lr}
       |C_a|+1, &  a=q-1\\
       |C_a|-1, &  a= r-1\\
       |C_a|, &  \text{otherwise}
     \end{array}
   \right.
 \]
\end{lemma}
\begin{cor}
With $n=bq+r, 0\leq a\leq q-1$ as above, and $0<k<r$.  
 \begin{equation}
   |C_{a+k}| = \left\{
     \begin{array}{lr}
       |C_a|+1, &  q-k\leq a\leq q-1\\
       |C_a|-1, &  r-k\leq a\leq r-1\\
       |C_a|, &  \text{otherwise}
     \end{array}
   \right.
\end{equation}
\end{cor}
Regarding $\tau_s$ within a specific translation class $[\tau]\in\mathcal{T}$ we revisit the $+1$ action on the $(n,q)$-cycles 
\[C_{0}, C_{-r},\ldots, C_{-(q-1)r}.\]
The last element of $C_{r-1}$ is $bq+r-1 =n-1$ and the first element of $C_{q-1}$ is $q-1=n-1+q \mod n$. The $q$ cycles $C_{r-1},C_{q-1}$ appear consecutively in $\mathcal C$. By lemma \ref{lem:strlen} the $+1$ action takes $C_{r-1},C_{q-1}$ to $C_{r},C_{0}$ exchanging their lengths. Consequently the $+1$ action on the $q$-tuple of lengths sends \[|\mathcal C| \to (|C_{1}|,|C_{-r+1}|\ldots,|C_{-(q-1)r+1}|).\]
This action stabilizes all but the $2$ components $C_{r-1},C_{q-1}$ transposing the lengths $|C_{r-1}|=b+1,|C_{q-1}|=b$ within $|\mathcal{C}|$. This is another way to see that $\mathcal{C}$ has exactly $q$ elements in any translation class $[\tau]$ since there are $q$ instances when the $+1$ action swaps pairs of adjacent $q$ cycles lengths. The next theorem is a far more elegant alternative to the brute force algorithm presented at the end of section \ref{sec:qcycles}.  
\begin{thm}
\label{thm:digstr}
A periodic digit string encodes a rational nonsingular billiard orbit in $\Delta$ if and only if it has a block pattern consisting of two lengths $b,b-1$ and has the property that some nontrivial power of a cyclic permutation applied to the list of block lengths differs by adjacent transposition.  
\end{thm}
\begin{proof}
By \ref{thm:Umbledone} it suffices to prove the theorem at the level of $q$-cycles. Write
\[\mathcal C=(C_{0},C_{-r}\ldots,C_{-(q-1)r})\]
for this ordered partition of $kq \mod n$, indices taken mod $q$.
Those $C_k$ which satisfy $0\leq k < r$ have length $b+1$ while $r \leq k < q$ have length $b$.
Thus the $+1$ action on $\Z/n\Z$ sends 
\[0,q,2q\ldots, (n-1)q \mod n\] to 
\[1,q+1,2q+1\ldots, (n-1)q+1 \mod n\] 
and
\[0,-r,-2r,\ldots -(q-1)r\] 
to 
\[1,-r+1,-2r+1,\ldots -(q-1)r+1.\] 
Consequently, $+1$ acts by some power $\sigma^p$ of the cyclic permutation $\sigma$ on the ordered partition of $q$-cycles $\Z/n\Z$ by mapping
\[C=(C_{0},C_{-r}\ldots,C_{-(q-1)r})\]
to
\[C=(C_{1},C_{-r+1}\ldots,C_{-(q-1)r+1})\]
where $p$ is a solution to $-pr+1=0 \mod q$. In other words, $p$ is the multiplicative inverse of $-r$ in $\Z/q\Z$.
\end{proof}
Within each block the digits are constructed accordingto the $+i,+j$ rule allowing for a periodic  string to be checked by the following algorithm:
\begin{enumerate}
\item Find the period $p$
\item Check that $d$ has two blocks of lengths $b,b+1$ according to the ``consecutive $+i$ then $+j$ rule.
\item Consider the sequence of block lengths $|\mathcal{C}|$
\item Repeat the process for the +1 action on $d$, obtaining a second sequence of block lengths $\mathcal{C}'|$
\item If $|\mathcal{C}'|$ $|\mathcal{C}|$ differ by cyclic permutation, then $|\mathcal{C}|$ is some $(q,n)$ string of a billard sequence.  
\end{enumerate}
\section{Examples}
\label{sec:examples}
\subsection{$\Delta$ is an equilateral triangular table with Weyl group $A_2$}
Label the vertices of $\Delta$ as $x_0,x_1,x_2$ so that edge $i$ connects vertex $x_i$ to $x_{i+1}$ mod $3$. The coordinate directions $v_1 =x_1-x_0$ and $v_2=x_2-x_0$, are called rhombic coordinates in \cite{BU08} so that
\begin{equation}
\label{equ:initdir}
v_0=a(x_1-x_0)+(1-a)(x_2-x_0)
\end{equation}
Recall that $H\unlhd V(\Lambda)$ is the translation subgroup of the unlabeled lattice and $H_l$ is the subgroup of the labelled lattice. $H_l$ is an index $9$ subgroup of $H$ generated by $3(x_1-x_0),3(x_2-x_0)$ for which $R^2/H$ is a torus tiled by $18$ unfoldings of $\Delta$ (as the edge labels on parallel sides are equal) (see fig. \ref{fig:MSA2}).
\begin{figure}
\label{fig:MSA2}
\includegraphics[scale=.6]{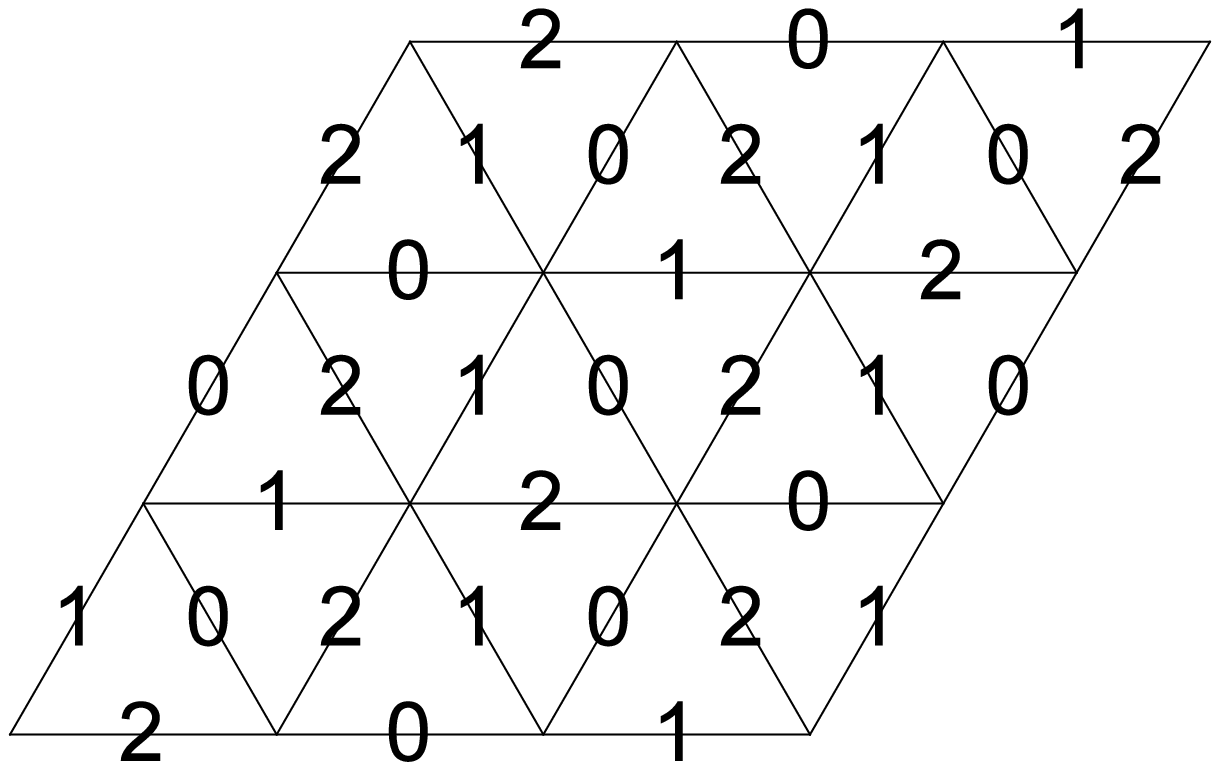}
\includegraphics[scale=.6]{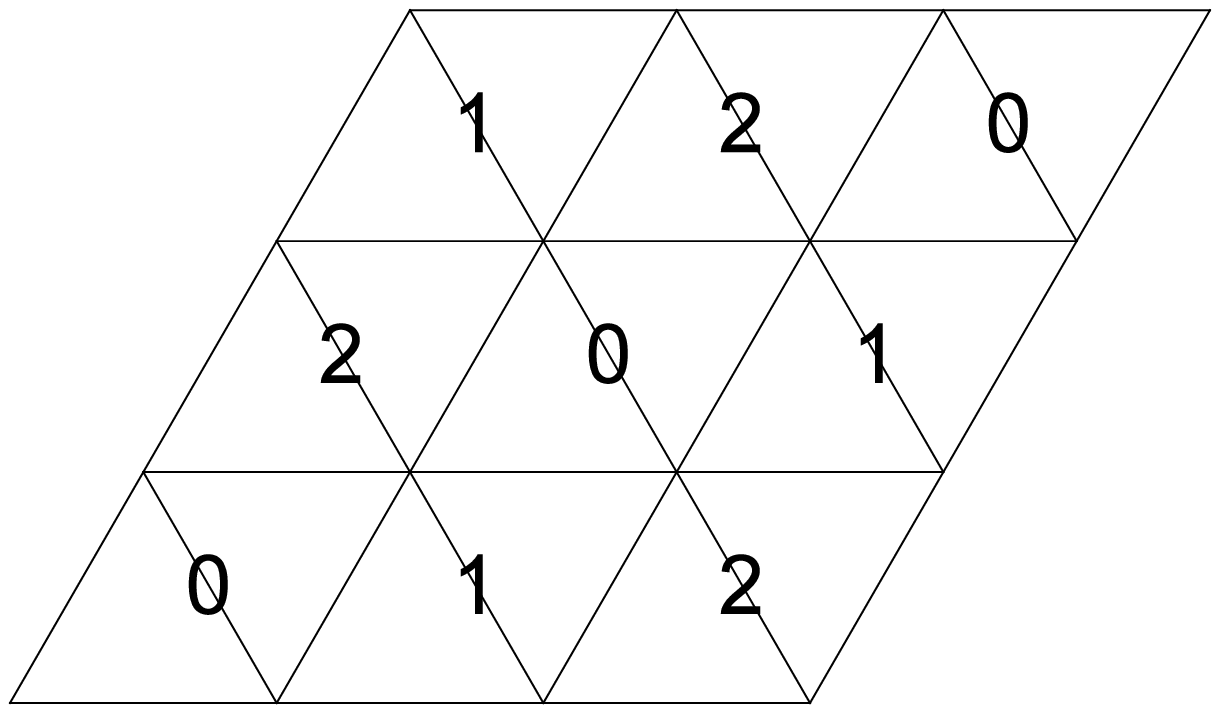}
\caption{To the left is a fundamental region for $R^2/H$ tiled by $18$ copies of $\Delta$ under reflection. To the right is the fundamental region for $R^2/H$ with coordinate directions $x_2-x_0$, $x_1-x_0$ removed.}

\end{figure}
Our key observation follows from direct inspection: any nonsingular trajectory $\tau$ in $\Lambda$ with direction $v_0$ as in \eqref{equ:initdir} intersects an edge in the direction $x_2-x_1$ followed by an edge in either of the directions from $E=\{(x_2-x_0),(x_1-x_0)\}$ in alternating sequence. Thus, by removing the edges from $\Lambda$ in the coordinate directions from $E$ removes every other digit from $\tau_s$. The resulting digit string $\overline{\tau_s}$ falls into the general scheme outlined in \S\ref{sec:calcs} above with $v_1\ (+1),v_2\ (-1)$ defined thusly. One may readily recover $\tau_s$ from $\overline{\tau_s}$ by inserting between each pair of distinct adjacent digits $i,j$ $i\neq j$ in $\overline{\tau_s}$ the missing digit from $\{0,1,2\}/\{i,j\}$. Thus the main theorems \ref{thm:digstr} and \ref{thm:Umbledone} apply to the ternary strings of $\overline{\tau_s}$. 
\begin{example}
$n=23 , q=5$\\
We have
\begin{align*}
n &= bq + r \\
23 &= 4\cdot 5+3 
\end{align*}
The $5$-orbits mod $23$ are
\begin{align*}
C_0 &=(0,5,10,15,20)\\
C_2 &=(2,7,12,17,22)\\
C_4 &=(4,9,14,19)\\
C_1 &=(1,6,11,16,21)\\
C_3 &=(3,8,13,18)\\
\end{align*}
The corresponding $(n,q)=(23,5)$ digit string is
\[01201\ 01201\ 0120\ 20120\ 2012\]
\[12012\ 12012\ 1201\ 01201\ 0120\]
\[20120\ 20120\ 2012\ 12012\ 1201\]
A billiard string inserts between every pair of digits the missing digit mod $3$. Thus, the billiard string in the direction $v_0=18/23v_1+5/23v_2$ is 
\[0210210212\ 0210210212\ 02102101\ 2102102101\ 21021020\]
\[1021021020\ 1021021020\ 10210212\ 0210210212\ 02102101\]
\[2102102101\ 2102102101\ 21021020\ 1021021020\ 10210212\]
Notice that the $q$-orbits 
\begin{equation}
(C_0,C_2,C_4,C_1,C_3)
\end{equation}
fall into the arithmetic progression $0,2,4,1,3$ of $-rk \mod q$. Therefore, the $2$-orbits mod $5$ are
\begin{align*}
C_0 &= (0,2,4)\\
C_1 &= (1,3)
\end{align*}
and the digit string of $(n,q) = (5,2)$ is
\[012\ 12\ 120\ 20\ 201\ 01\]
The statement of Theorem \ref{thm:strED} says the $(23,5)$-cycle can be built from
the $(5,2)$ string by replacing each number $y\in C_i$ by its corresonding $5$-cycle $C_y$. This is evident in the calculation above.       
\end{example}
\subsection{$\Box$ is a square table with Weyl group $D_2$}
As in $B_2$, we restrict our focus to initial directions $v_0 = (q/n, (n-q)/n)$ with $q<n/2.$ With the edges of $\Box$ labelled $0,1,2,3$, removing all edges of $\Lambda$ labelled $2,3$ yields a square lattice twice the size as $B_2$ (see fig. \ref{fig:MSD2}). Thus the classification theorems \ref{thm:Umbledone} \ref{thm:strED} apply for the $D_2$ billiard strings by means of the $(n,q)$ classes according to the ``$+0$ within $+1$'' between classes rule, modulo $2$, These are also known as \emph{Sturmian sequences} \cite{SU10}. 
We state the classification for $D_2:$
\begin{thm}
\label{thm:strD_2}
Write $n=bq+r$ The billiard strings of $D_2$ are in bijective correspondence with $(n,q)$-binary strings $d\in \{0,1\}^{\infty}$ built from two blocks of length $b,b+1$ of repeating digits. Adjacent blocks have opposite digits and the arrangement of blocks follow from the orbit of $-r,-2r,\ldots \mod q$.
\end{thm}
\begin{figure}
\label{fig:MSD2}
\includegraphics[scale=.4]{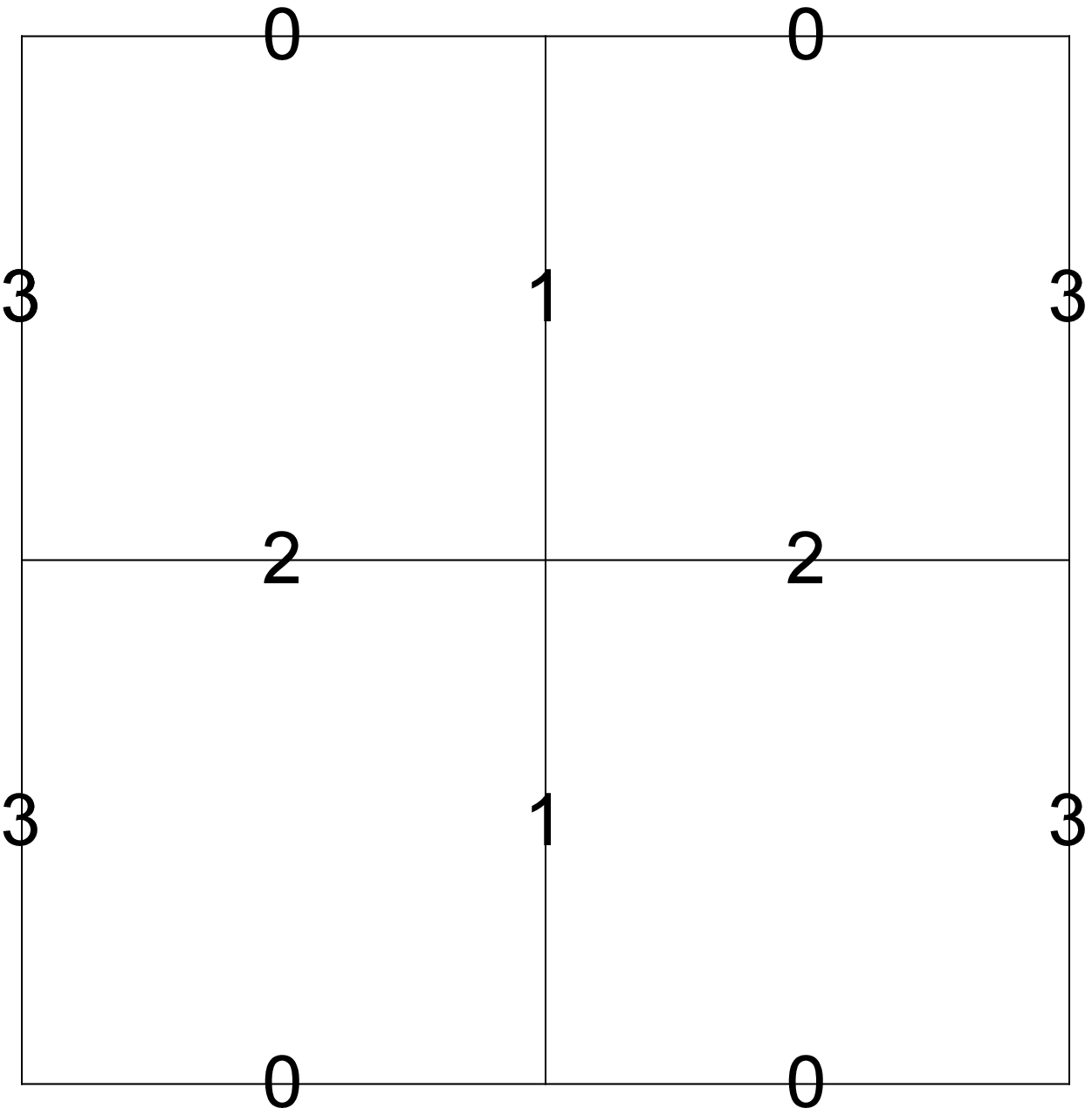}
\includegraphics[scale=.4]{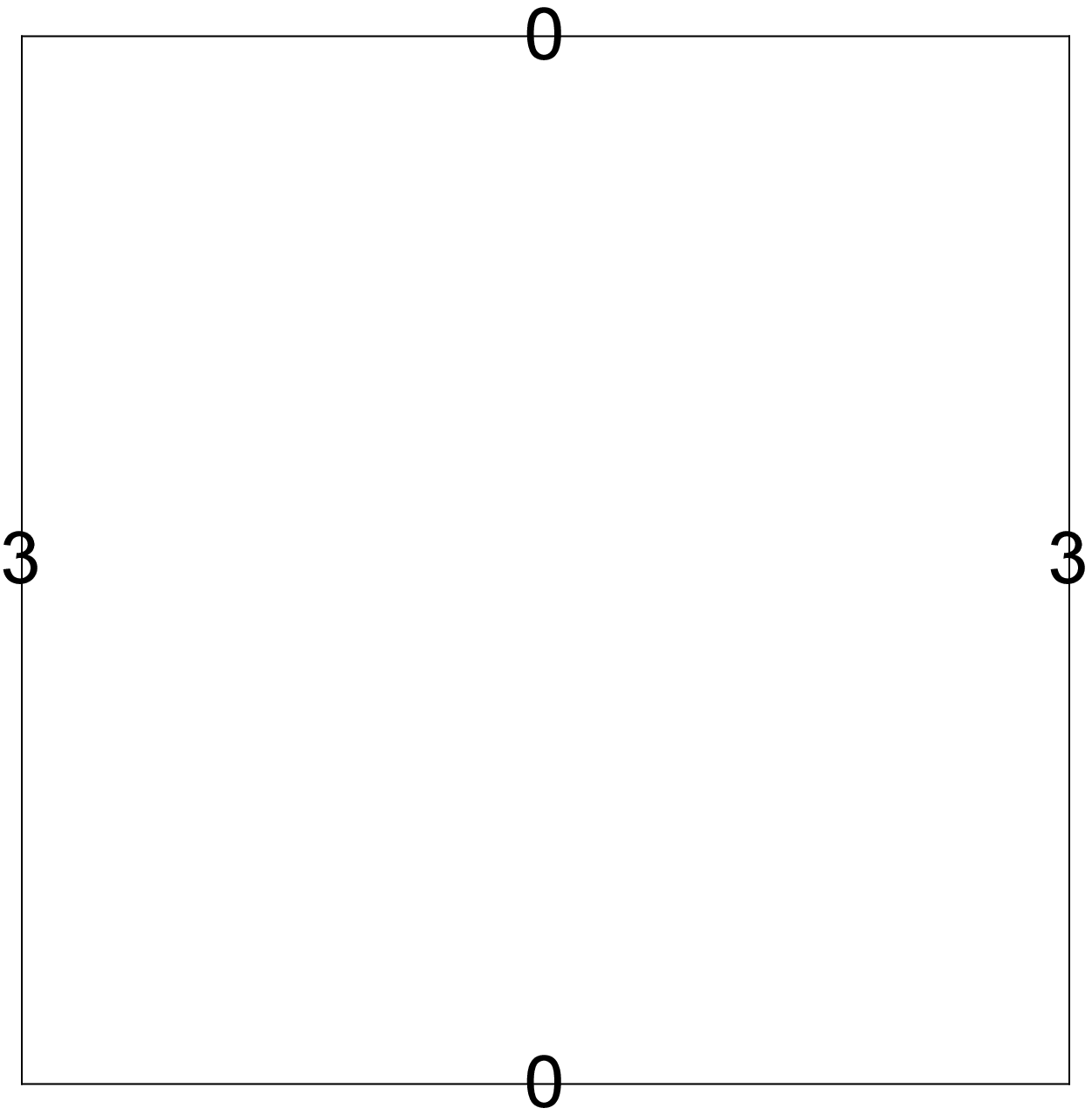}
\caption{Torus $\R^2/H$ for $D_2$, tiled by $4$ copies of $\Box$ under reflection. Torus $\R^2/H$ for $D_2$ with edges removed. The substiution $3\to 1$ yields the classical Sturmian sequences}
\end{figure}
Notice that between any pair of adjacent digits in $0,3$ is either $1,12,21,2$. 
If we remove all edges $0,3$ and apply \ref{thm:strD_2} we obtain another string in the digits $12$. If the intial position $p_0$ and $v_0$ are equal then the two sequnces are shuffled together by alternating between blocks from the $0,3$ and $1,2$ sequences.
\subsection{$\Delta$ is a $45-45-90$ trianglular table with Weyl group $B_2$}
See figure $\ref{fig:MSB2}$. Notice that every other digit of a $B_2$ biliard string is $0$. Assume every even digit of $\tau_s$ is $0$; removing them yields a binary string $\tau_s$ in the digits $1$ and $2$ (see fig. \ref{fig:MSB2})). We prefer to use the binary digits $0,1$. Starting with edge $p_0 = (0,t), t\in [0,1]$ labelled $2$ in the figure, we note that a unit step up (in $(1,0)$) has the same digit as one which preceds it and a unit step right has the opposite digit. Thus we consider the truncated billiard string $\overline{\tau_s}$ where the unit steps $v_1$, $v_2$ append a digit that is $+0,+1$ to the preceding digit. The main theorems \ref{thm:Umbledone} \ref{thm:strED} apply for $(n,q)$-classes according to the $+1$ within classes, $+0$ between classes rule, mod $2$. We state the classification for $B_2:$
\begin{thm}
Write $n=bq+r$ The billiard strings of $B_2$ are in bijective correspondence with $(n,q)$-binary strings $d\in \{0,1\}^{\infty}$ built from two blocks of length $b,b+1$ of alternating digits. Adjacent blocks share the same digit and the arrangement of the blocks follow from the orbit of $-r,-2r,\ldots \mod q$.  
\end{thm}
\begin{figure}
\label{fig:MSB2}
\includegraphics[scale=.4]{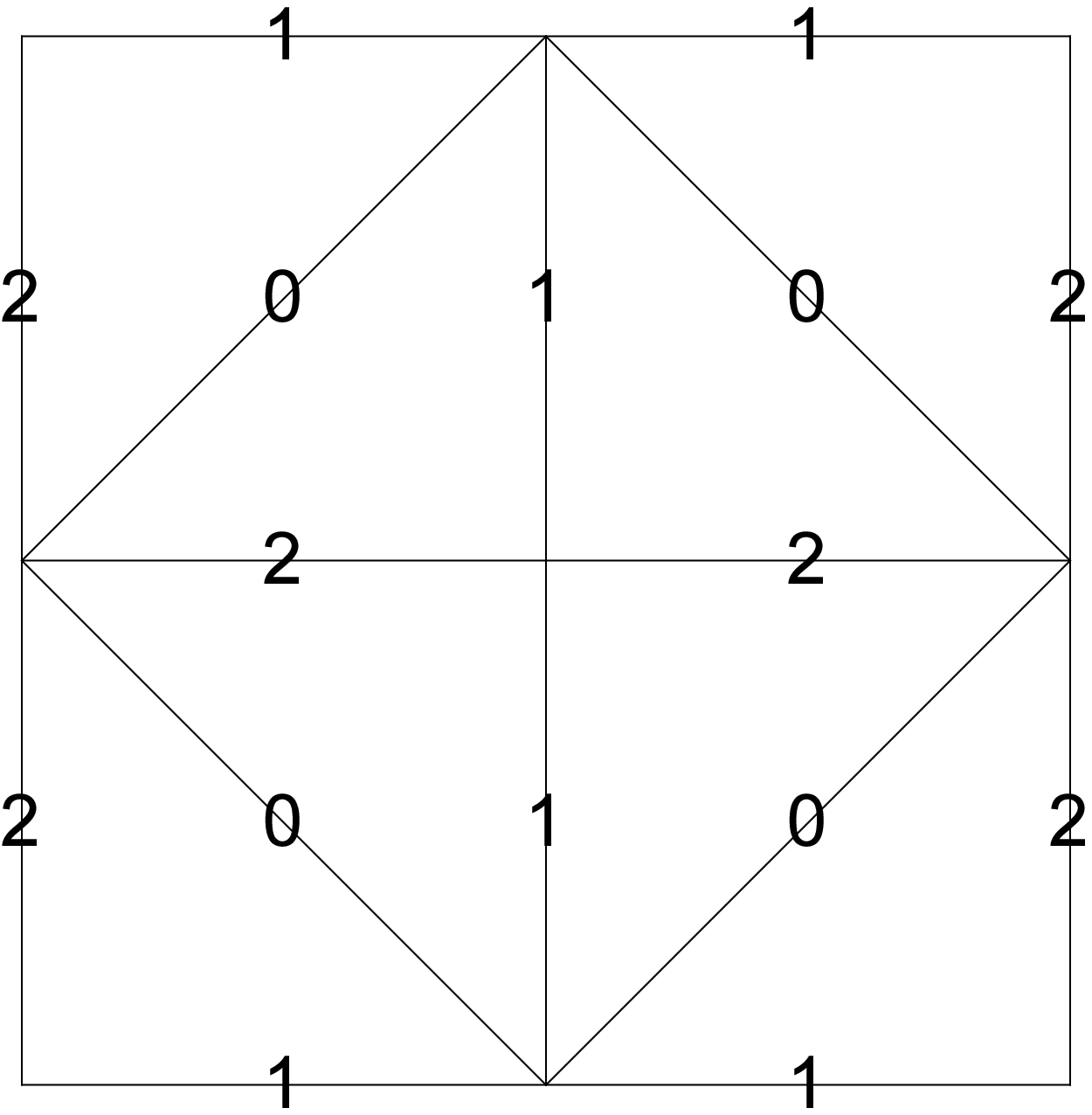}
\label{fig:MSB2r}
\includegraphics[scale=.4]{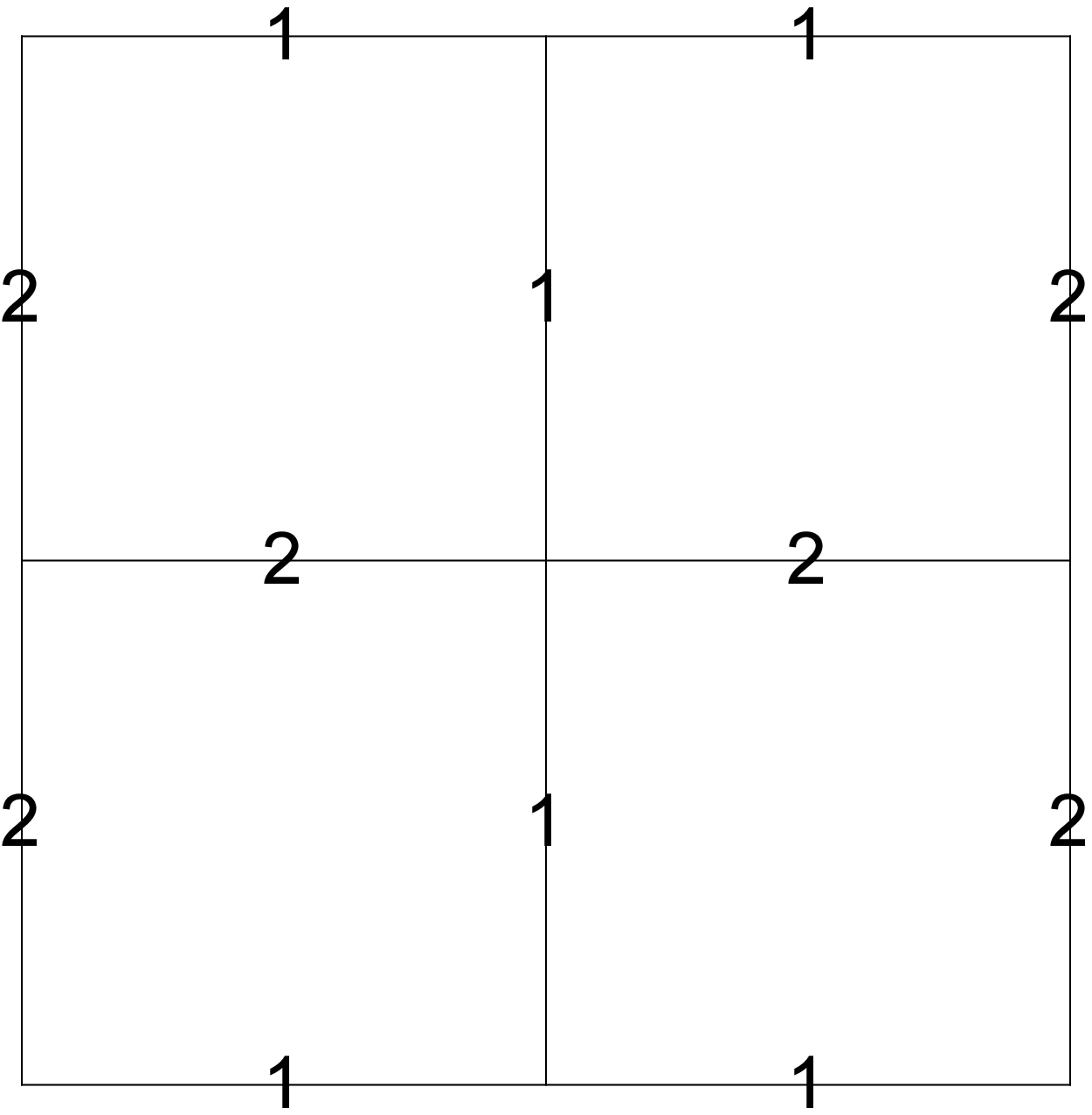}
\caption{To the left is the torus $\R^2/H$ for $B_2$ which tiled by $8$ copies of $\Delta$ under reflection. To the right is the torus $\R^2/H$ for $B_2$ with diagonal edges removed.}
\end{figure}
\subsection{$\Delta$ is a $30-60-90$ table with Weyl group $G_2$.}
See figure \ref{fig:MSG2}. Label the edges of $\Delta$ $0,1,2$ by increasing length. Removing the long edges (labelled $2$) from $\Lambda$ yields a rectangular lattice. There is an linear map taking the rectangulr lattice to the square lattice $D_2$ which distorts $t,\tau$ but leaves the digit sequence $\overline{\tau_s}$ invariant. Thus the $G_2$ billiard sequences are classified by $D_2$ strings above by shuffling the $0,1$ string with the constant sequence $2,2,\ldots$.
\begin{figure}
\label{fig:MSG2}
\includegraphics[scale=.4]{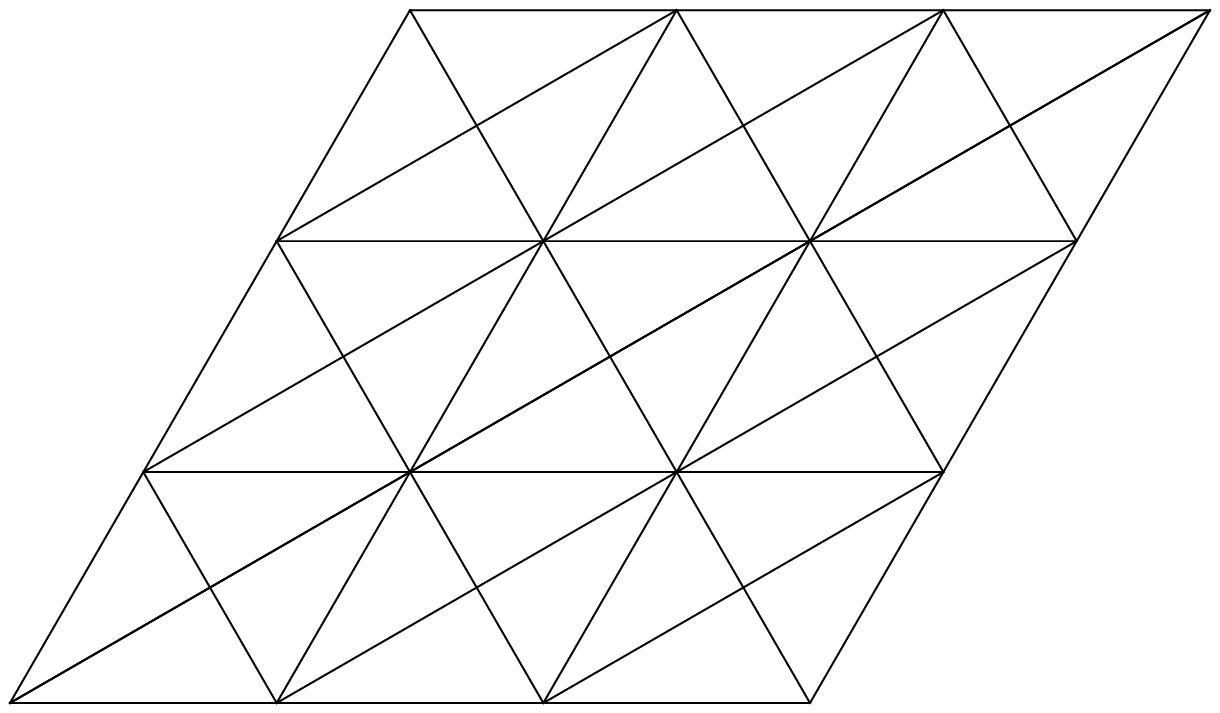}
\includegraphics[scale=.4]{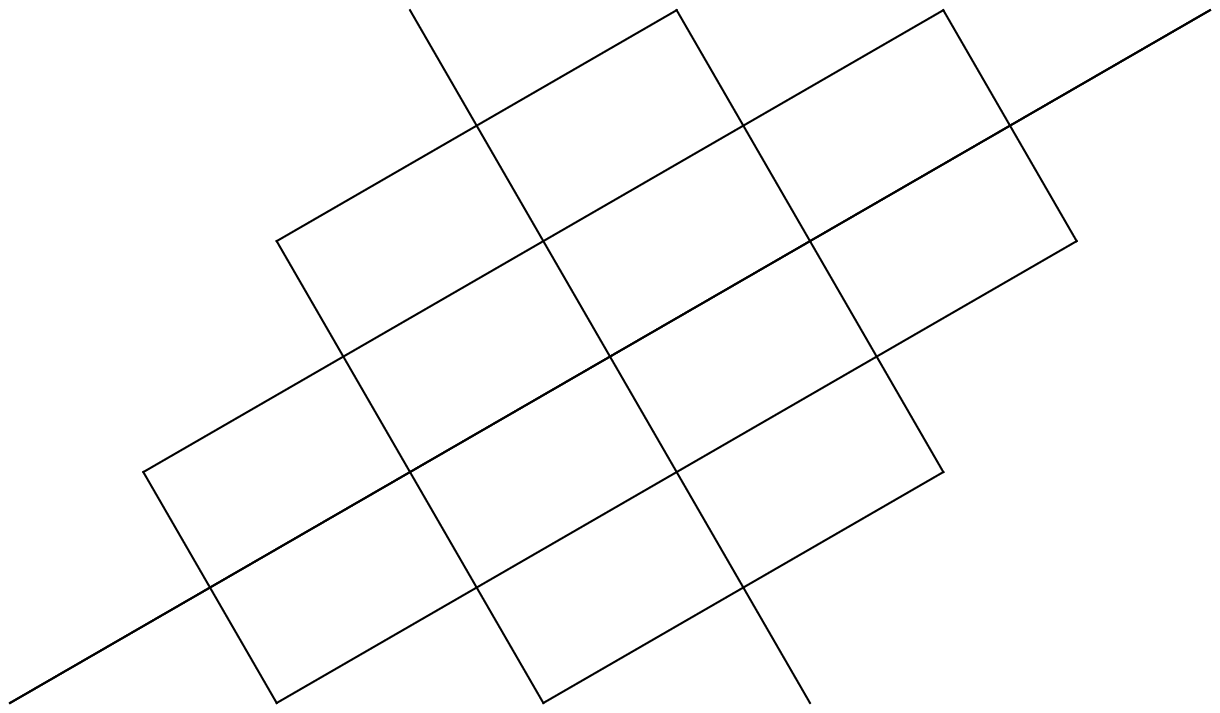}
\caption{Torus $\R^2/H$ for $G_2$, tiled by $4$ copies of $\Delta$ under reflection. Torus $\R^2/H$ for $D_2$ with edges removed. 9 fundamental regions are shown to reveal the underlying rectangular lattice.}
\end{figure}
\section{Aperiodic trajectories}
\label{sec:aperiodic}
Lastly we comment on aperiodic trajectories. By this we mean a billiard sequence which arises from an irrationally sloped line in $\R^2$. As a consequence of Minkowsk's Theorem, any $\epsilon$ band around $\tau$ passes through an integer point in $\Z^2$, and it is exactly at this point when the corresponding digit string swaps digits. While the explicit digit string is difficult to characterize, we can say that the corresponding translation class has an infinite number of elements since each element from a class is within $\epsilon$ to an integer point, for all $\epsilon>0$. Since each ray may be viewed as a winding line on a torus, $|[\tau]|$ infinite follows from $\tau$ dense on $\R^2/H.$ 
\section{Future research}
	Future research could extend the methods in this paper to billards in higher dimensions, i.e., billiard trajectoies in tables of higher rank affine Weyl groups $A_n,B_n$ (dual to $C_n$) $D_n$ $F_4, E_6,E_7,E_8$. The digit sequnces that arise should exhibit similar interesting number-theoretic properties. These ``tables'' would tessellate space by reflection along facets. A more general discussion of digit sequences on various tilings in $\R^2$ can be found in \cite{Da13}. Sturmian sequences for $2n$-gons were studied extensively in \cite{SU10}.

\end{document}